\documentclass[11pt]{article}
\usepackage{latexsym,euscript,amsmath,amssymb,amsbsy,amsfonts,amsthm,amsopn,amstext,amsxtra,amscd}
\usepackage{epsfig}
\usepackage{graphics}
\usepackage{url}
\usepackage{enumerate}
\usepackage[english]{babel}

\theoremstyle{plain}
\newtheorem{theorem}{Theorem}

\newtheorem{corollary}{Corollary}
\newtheorem{proposition}{Proposition}

\theoremstyle{definition}
\newtheorem{example}{Example}

\newtheorem{remark}{Remark}

\newcommand{\N}{{\mathbb N}}

\newcommand{\D}{\text{$\mathcal{D}$}}

\newcommand{\MM}{\mathbb M}
\newcommand{\NN}{\mathbb N}

\newcommand{\A}{\mathbb A}

\title{Uniform distribution with respect to density}

\author{Ligia L. Cristea \thanks{This author is supported by the Austrian Science Fund (FWF),
Project P27050-N26 and by the Austrain-French cooperation project FWF I1136-N26.}
\\Karl-Franzens-Universit\"at Graz\\ Institut f\"ur Mathematik und Wissenschaftliches Rechnen
\\Heinrichstrasse 36, 8010 Graz,\\Austria\\ \tt{strublistea@gmail.com}
\and Milan Pa\v{s}t\'eka \thanks{This author is supported by the grant VEGA 2/0146/14.}\\Trnavska Univerzita, Pedagogick\'a Fakulta, \\ Priemyseln\'a 4
 P. O. BOX 9 , 918 43  Trnava
 \\Slovakia\\ \tt{pasteka@mat.savba.sk}  }

\begin{document}
\maketitle

\abstract{The paper deals with a generalisation of uniform distribution. The analogues of Weyl's criterion are derived. }


\section{Introduction}\label{sec:introduction}
The notion of uniformly distributed sequence was for the first time defined and studied by Herman Weyl in 1916 in his paper [W]. Since then the theory of uniform distribution was developed in a lot of directions by numerous authors. For a survey we refer to the monographs [HLA], [K-N], [D-T], [P-S].
Since a sequence is a mapping defined on the set of positive integers, the uniform distribution is based on the asymptotic density of coimages of intervals, as defined bellow.

Let $\N$ be the set of non-negative integers. If for some
 $A \subset \N$ there exists the limit
 $$
 \lim_{n\to \infty} \frac{1}{n} \big|\{ k \le n: k \in A \}\big| := d(A),
 $$
 we say that $A$ has asymptotic density, and the value $d(A)$ is called the {\it asymptotic density} of $A$. Denote by $\D$ the system that consists of the sets having asymptotic density (see [PA]). A sequence  $\{x_n\}$ of elements of the interval $[0,1]$ is \emph{uniformly distributed modulo $1$} if and only if
 for each subinterval $I \subset [0,1]$ the set $\{n; x_n \in I\}$ belongs to $\D$ and its asymptotic density is equal to the length of the interval $I$.

The aim of this paper is to describe uniformity of distribution with respect to a larger class of finitely additive probability measures defined on certain systems of sets.

In Section \ref{sec:preliminaries} we introduce the notions of density and $\pi$-uniformly measurable mapping with respect to a density $\pi$. We give some examples of $\pi$-uniformly measurable mappings defined on different spaces, and we prove a necessary and sufficient condition for the existence of a $\pi$-measurable mapping.

In the third section, dedicated to Riemann integrability and the Weyl criterion, we deal with real valued functions defined on a dense subset $\MM$ of a compact metric space $\Omega$ equipped with a Borel probability measure. We prove a necessary and sufficient condition for Riemann integrability in terms of uniformly distributed sequences and finally derive a result in terms of asymptotic density.

In  Section \ref{sec:set_of_non-negative_integers}  the role of the set $\MM$ from the previous section is played by the set of non-negative integers $\N$. First, we prove results that relate $\pi$-uniform measurability  to uniform distribution modulo $1$. In the sequel, the notions and results are extended from real valued mappings to mappings with values in an arbitrary compact metric space endowed with a  Borel probability measure.


\section{Preliminaries}
\label{sec:preliminaries}

Let $\MM$ be an arbitrary set and $\mathcal{Y}$ a system of subsets of $\MM$. Then $\mathcal{Y}$ is called a $q$-\emph{algebra of sets} if the following three conditions are satisfied:
i) $M \in \mathcal{Y}$, \newline
ii) $A, B \in \mathcal{Y}, A \cap B = \emptyset \Longrightarrow A\cup B \in
\mathcal{Y}$, \newline
iii) $A, B \in \mathcal{Y}, A \subset B  \Longrightarrow B\setminus A \in
\mathcal{Y}$.

It is well known that if in ii) the condition $A \cap B = \emptyset$ is omitted and in iii)
the condition $A \subset B$, then $\mathcal{Y}$ is called an {\it algebra of sets}.

$$ $$
Let $\pi$ be a finitely additive probability measure on the $q$-algebra $\mathcal{Y}$ fulfilling the condition

$A \in \mathcal{Y}$ if and only if for arbitrary $\varepsilon > 0$ there exist $A_1, A_2 \in \mathcal{Y}$
such that
\begin{equation}
\label{def_density}
A_1 \subset A \subset A_2, \pi(A_2) - \pi(A_1) < \varepsilon.
\end{equation}
We call such a finitely additive probability measure {\it density}.

A mapping $x: \MM \to [0,1]$ is called $\pi$-{\it uniformly measurable} if
for each subinterval $I \subset [0,1]$ the set $x^{-1}(I)$ belongs to $\mathcal{Y}$ and
\begin{equation}
\pi\left(x^{-1}(I)\right)=|I|.  \label{def_unif_measurable}
\end{equation}
The property \eqref{def_density} provides that $x$ is $\pi$-uniformly measurable if and only if \eqref{def_unif_measurable} holds for a set of intervals $I$, who's set of endpoints is dense in the unit interval.
\vskip0,5cm
For a better illustration we show some examples.

\begin{example}
Let $\MM =\N$ and $\{c_n\}$ be a sequence of positive real numbers with
$\sum_{n=1}^\infty c_n = \infty$.
 If for $A \subset \N$ there exists the limit
 $$
 \lim_{n\to \infty} \frac{\sum_{ k \le n: k \in A }c_k}{\sum_{k=1}^n c_k} := d_c(A),
 $$
 then we say that $A$ has weight density with respect to the sequence $\{c_n\}$.
 In the case $\mathcal{Y} =\D_c$, the algebra of the sets having corresponding weight density $d_c$ (see [PA]), and $\pi =d_c$,  the $\pi$-uniformly measurable mapping is a sequence uniformly distributed with respect to $d_c$.
\end{example}

\begin{example}
For $\MM =\N$, $\D_u$ the algebra of the sets having uniform density $u$ (see [PA]), and $\pi =u$,  the $\pi$-uniformly measurable mapping is a well distributed sequence.
\end{example}

\begin{example}
Let $\MM = [0,\infty)$. Denote by $\mathcal{Y}$ the family of all Lebes- gue measurable sets $S \subset \MM$ such that there exists the limit
\begin{equation*}
\pi(S) = \lim_{T \to \infty}\frac{\lambda(S \cap [0,T])}{T}.
\end{equation*}
Here the $\pi$-uniformly measurable mapping coincides with continuous uniform distribution
(see [K-N], [D-T], [SP]).
\end{example}

We conclude with an example which is a special case of the object of study in the next section.

Throughout the wohle paper, for any subset $A$ of a topological space, $cl(A)$ will denote the topological closure of $A$, and $Int(A)$ its interior in the given space.

\begin{example}
Consider once more $\MM =\N$ and $\Omega$ the ring of polyadic integers, (see [N], [N1]). Then the set function $\mu^\ast(S) = P(cl(S))$, where $P$ is the Haar measure on $\Omega$, is called Buck's measure density, (see [BUC], [PA]). If $\mathcal{Y}=\D_\mu$ is the algebra of all sets
measurable in the sense of Buck and $\mu$ the restriction of $\mu^\ast$ on $\D_\mu$ we get Buck's uniformly distributed sequences (see [PA], [PA1]). Let us remark that in the paper [P-P]   more general cases of Borel probability measures on the ring of polyadic integers are constructed, and these can induce densities on $\N$ in an analogous way. Similar examples of density are studied in [P-T] in the case when $\MM$ is a Dedekind domain with the finite norm property.
\end{example}

\vskip0,5cm

The definition of a $\pi$-uniformly measurable mapping leads to the question regarding the existence of such a mapping. We shall characterise it by the following property.

We say that $\pi$ has the {\it Darboux property} on $\mathcal{Y}$  if for every
$A \in \mathcal{Y}$ and any non negative $\alpha \le \pi(A)$ there exists $B \subset A$,
$B \in \mathcal{Y}$  with $\pi(B)=\alpha$.

The following result can be proven analogously to Theorem $2.2$ in [PS].


\begin{proposition} The following statements are equivalent:
\label{prop:equivalent_statements_darboux}
\begin{enumerate}
 \item $\pi$ has the  Darboux property on $\mathcal{Y}$.
\item For every $B \in \mathcal{Y}$ there exists $B_1 \subset B$ such that $B_1 \in \mathcal{Y}$ and $\pi(B_1) = \frac{1}{2}\pi(B)$.
\item For arbitrary $\varepsilon>0$ there exist the sets $C_1,\dots,C_k$ such that
$\NN = C_1\cup\dots\cup C_k$ and $\pi(C_j) < \varepsilon$, $j=1,\dots,k$.
\end{enumerate}
\end{proposition}


\begin{theorem}
\label{theo:unif_meas_exists_iff_darboux}
A $\pi$-uniformly measurable mapping exists if and only if $\pi$ has the Darboux property on $\mathcal{Y}$.
\end{theorem}

\begin{proof}
The condition 2 in Proposition \ref{prop:equivalent_statements_darboux} provides that there exists a system of decompositions of $\MM$
$\{U(j,2^n)$, $j=0,\dots,2^n-1\}, n=1,2,\dots$, such that $\pi(U(j,2^n))=\frac{1}{2^n}$
and $U(j,2^n)= U(2j,2^{n+1})\cup U(2j+1,2^{n+1})$.

In the same manner we can consider a system of closed intervals
$I(j,2^n)$
$=[\frac{j}{2^n},\frac{j+1}{2^n}]$, for $j=0,\dots,2^n-1,~~n=1,2,\dots$. For each $k\in \N$ we have a uniquely determined system $U(j_k^n,2^n), n=1,2,\dots$, such that $k \in U(j_k^n,2^n)$. The corresponding intervals
$I(j_k,2^n)$ form a centered system of closed sets, thus its intersection is nonempty. Let us denote by
$x(k)$ the element of $\cap_{n=1}^\infty I(j_k^n,2^n)$. Then we have a mapping
$x:\MM \to [0,1]$ such that
$$
k \in U(j,2^n) \Longrightarrow x(k) \in I(j,2^n), k \in \MM.
$$
Thus  $U(j,2^n) \subset x^{-1} (I(j,2^n))$.
Clearly, if $x(k) \in Int(I(j,2^n))$ then
$k \in U(j,2^n)$ and thus $x^{-1}(Int(I(j,2^n))) \subset U(j,2^n)$, where $Int$ denotes the interior with respect to the topology induced by the Euclidean metric.

Each $y \in [0,1]$ can be contained in some interval $I(j,2^n)\cup I(j+1,2^n)$. Therefore
$ x^{-1}(\{y\}) \subset U(j,2^n) \cup U(j+1,2^n), n \in \N$. This yields
$\pi(x^{-1}(\{y\}))=0$, and thus
$\pi(x^{-1} (\{\frac{j}{2^n},\frac{j+1}{2^n}\})) =0$. Considering the inclusions
$$
U(j,2^n) \setminus x^{-1}\Big(\{\frac{j}{2^n},\frac{j+1}{2^n}\}\Big) \subset x^{-1}\big(Int(I(j,2^n))\big) \subset U(j,2^n)
$$
we have $x^{-1}\big(Int(I(j,2^n))\big) \in \mathcal{Y}$ and $\pi\left(x^{-1}(Int(I(j,2^n)))\right)=
2^{-n}$. Since the set of endpoints of intervals $I(j,2^n)$ is dense in the unit interval we thus have proven that
$x$ is $\pi$-uniformly measurable.

The existence of a $\pi$-uniformly measurable mapping provides by
Proposition 1 that $\pi$ has the Darboux property on $\mathcal{Y}$.
\end{proof}

\section{Riemann integrability and the Weyl criterion}\label{sec:riemann_integrability_weyl_criterion}
In this section the generalisation of Example 4 will be studied, as announced above. We shall suppose that $\MM$  is a dense subset of a compact metric space $(\Omega, \rho)$ equipped with a Borel probability measure $P$.

A sequence $\{a_n\}$, $a_n \in \Omega$ is called \emph{uniformly distributed in $\Omega$} if
$$
\lim_{N \to \infty} \frac{1}{N} \sum_{n=1}^N f(a_n) = \int f dP,
$$
for every continuous real function $f$ defined on $\Omega$. Equivalently, this means that
for every measurable
$C \subset \Omega$ with $P\big(cl(C)\setminus Int(C)\big) =0$ the set of indices $\A( \{a_n\}, C):=\{n; a_n \in C\}$ has asymptotic density equal to $P(C)$, (see [D-T], [K-N], [S-P]). Here $ cl(C)$ and $ Int(C)$ denote the closure and, respectively, the interior of the set $C$ with respect to the topology on $\Omega$ that is induced by the metric $\rho$.

A measurable set $C \subset \Omega$ with $P(cl(C)\setminus Int(C)) =0$ is called \emph{set of}
$P$-\emph{continuity} (see [D-T], [K-N], [S-P]).

We shall consider the set function
\begin{equation}
\label{pi_ast}
\pi^\ast(S) = P(cl(S)).
\end{equation}

It can be easily checked that $\pi^\ast$ is a strong submeasure on the system of all subsets of $\MM$. This yields that the system of sets $\mathcal{Y}$, consisting of all $A \subset \MM$ with
$\pi^\ast(A)+ \pi^\ast(\MM \setminus A)=1$, is an algebra of sets and the restriction
$\pi = \pi^\ast|_\mathcal{Y}$ is a finitely additive probability measure on this algebra, fulfilling
the condition \eqref{pi_ast}.

An important role will be played by Riemann integrability and the Riemann integral of bounded real valued functions defined on $\MM$.

We rewrite the definition of the Riemann integral from [Bi] and we reprove the criterion from [Bi] for slightly more general conditions. Then we apply this result in the investigation of uniformly distributed sequences.

 Let $C(\Omega)$ be the set of all continuous real functions defined on $\Omega$.
For an arbitrary bounded real valued function $f$ defined on $\MM$ the values
$$
\int^\ast f = \inf \Big\{\int g dP; g \ge f, g \in C(\Omega) \Big\}
$$
will be called the \emph{upper Riemann integral of} $f$ and
$$
\int_\ast f = \sup \Big\{\int g dP; g \le f, g \in C(\Omega) \Big\}
$$
the \emph{lower Riemann integral of} $f$, respectively. The function $f$ will be called \emph{Riemann integrable} if
$$
\int^\ast f = \int_\ast f := \int f,
$$
in this case this value is called the {\it Riemann integral of} $f$.

As usually $B(x, r)$ will be the open ball with center $ x\in \Omega$ and radius $r>0$.
Since every system of disjoint sets of positive measure is countable we can conclude that for $x \in \Omega$ and
positive $r_1 < r_2$ there is $r$, $r_1 < r < r_2$ such that the ball $B(x, r)$ is a set of $P$-continuity.
From the compactness of $\Omega$ we can derive that for each $\alpha$ the set $\Omega$ can be covered
by a finite system of open balls with radius smaller than $\alpha$, and from these balls we can construct
disjoint sets of $P$-continuity $L_1,\dots,L_m$ such that
$$
\Omega= \bigcup_{j=1}^m L_j,
$$
where $diam(L_j) < \alpha$, for $ j=1,\dots,m$. Thus we can construct a system
$\mathcal{L}_n$, $n=1,2,\dots,$ of disjoint finite covers of $\Omega$ that consist of  $P$-continuity sets,
 $\mathcal{L}_n= \{L_1^{n},\dots,L^{n}_{m_n}\}$ with $diam(L_j^{n}) \to 0$ for $n \to \infty$,
uniformly with respect to $j$. These can be arranged in such a manner that every set from
$\mathcal{L}_n$ is a disjoint union of sets from $\mathcal{L}_{n+1}$.

Put $K_j^n =L_j^{n} \cap \MM, j=1,\dots,m_n, n=1, 2, \dots$. Denote by $\mathcal{F}$ the set of all
real functions of the form $f=\sum_{j=1}^{m_n} \beta_j \mathcal{X}_{K_j^n}$, $\beta_j\in \mathbb{R}$. These functions are Riemann integrable and $\int f =\sum_{j=1}^{m_n} \beta_j P(L_j^n)$. By applying Urysohn's Lemma to the sets of $P$-continuity $L_j^n$ we deduce that in the definition of the upper and lower Riemann integral the set
$C(\Omega)$ can be replaced by $\mathcal{F}$.

Thus a bounded function $h$ is Riemann integrable if and only if for arbitrary $\varepsilon >0$ there exist $f_1, f_2 \in \mathcal{F}$ such that $f_1 \le h \le f_2$ and $\int(f_2-f_1) < \varepsilon$.
Clearly, every uniformly continuous real valued function defined on $\MM$ is Riemann integrable.

There exists at least one sequence of elements of $\Omega$ that is uniformly distributed in $\Omega$
(see [H], [D-T], [K-N]). Since $\MM$ is dense in $\Omega$ we can suppose the existence of at least one uniformly distributed sequence
with elements from $\MM$.

Denote by $C(\MM)$ the set of all uniformly continuous real functions defined on this space. This set of functions coincides with the set of all restrictions of continuous real functions defined on $\Omega$ to $\MM$.  Thus a sequence $\{x_n\}$, $x_n \in \MM$, is uniformly distributed if and only if
$$
\lim_{N \to \infty} \frac{1}{N} \sum_{n=1}^N g(x_n) = \int g,
$$
for each $g \in C(\MM)$.

\begin{remark}
Let $R(\MM)$ be the set of all Riemann integrable real functions defined on $\MM$. Clearly,
$C(\MM) \subset R(\MM)$ and $\mathcal{F} \subset R(\MM)$. Moreover,
\begin{enumerate}[(i)]
  \item $R(\MM)$ is a vector space and $\int$ is a non-negative linear functional on this vector space, and
\item $R(\MM)$ is bounded with respect to the supremum metric and $\int$ is continuous with respect to this metric.
\end{enumerate}
\end{remark}

The relationship between uniform distribution and Riemann integrability in the unit interval is described in the paper by  de Bruijn und Post [BP]. Later this result was proven in [Bi] by a different method and extended for compact metric spaces. This method was the one exploited in [TW] for uniform distribution. We apply this procedure in our case.


\begin{theorem}
\label{theo:riemann_integrable_iff_exists_proper_limit}
Let $f$ be a bounded real valued function defined on $\MM$. Then $f$ is Riemann integrable if and only if for every uniformly distributed sequence $\{x_n\}$ , $x_n \in \MM$, there exists
the proper limit
$$
\lim_{N \to \infty} \frac{1}{N} \sum_{n=1}^N f(x_n);
$$
in this case this limit is equal to $\int f$.
\end{theorem}

\begin{proof}
One implication follows immediately from Weyl's criterion.

Suppose that $f$ is not Riemann integrable.
For $x \in \MM$, define the functions $H_n, h_n$,
for $ n=1,2,\dots$, as follows:
$$
H_n(x)= \sup \{f(y); y, x \in K_j^{n} \},
$$
$$
h_n(x)= \inf \{f(y); y, x \in K_j^{n} \}.
$$
Clearly $H_n, h_n \in \mathcal{F}$ and $h_n \le f \le H_n$, which implies
$$
\int h_n \le \int_\ast f < \int^\ast f \le \int H_n, \text{ for } n=1,2,\dots.
$$

Let $\{a_n\}$ be a fixed uniformly distributed sequence of elements of $\MM$. Then for every positive integer $s$ we have
$$
\lim_{N \to \infty} \frac{1}{N} \sum_{n=1}^N H_s(a_n) = \int H_s.
$$
This yields that for every $s$ there exists $N(s)$ such that
$$
\frac{1}{N} \sum_{n=1}^N H_s(a_n) > \int H_s- \frac{1}{s}, \text{ for } N \ge N(s),
$$
and we can suppose that $N(1) < N(2) < \dots< N(s)< \dots$. For every positive integer $n$ there exists an uniquely determined $s(n)$ such that $N(s(n)) \le n <N\big(s(n)+1\big)$. For any positive integer $n$, let
$$
H(a_n) = \sup \{H_i(a_n); i> s(n) \}.
$$
Then for every $n$ there exists $i(n) > s(n)$ such that $H_{i(n)}(a_n) \ge H(a_n)- \frac{1}{n}$. Suppose that
$a_n \in K_j^{i(n)}$, then we choose $y_n \in K_j^{i(n)}$ in order that
$f(y_n) \ge H_{i(n)}(a_n) - \frac{1}{n}$. From the definition of the sets $K_j^{i(n)}$ we see that
$\lim_{n \to \infty}\rho(y_n,a_n)$ $=0$, thus the sequence $\{y_n\}$ is uniformly distributed.

Now we can consider
$$
\frac{1}{N}\sum_{n=1}^N f(y_n) \ge
\frac{1}{N}\sum_{n=1}^N H_{i(n)}(a_n) -\frac{1}{N}\sum_{n=1}^N \frac{1}{n}.
$$
For $N > N(s(n))$ we have
$$
\frac{1}{N}\sum_{n=1}^N f(y_n) \ge \int H_{i(n)} -\frac{1}{i(n)}-\frac{1}{N}\sum_{n=1}^N \frac{1}{n}.
$$
Since the last two terms tend to 0, we get
$$
\liminf_{N \to \infty} \frac{1}{N}\sum_{n=1}^N f(y_n) \ge \int^\ast f.
$$
In an analogous way, only considering $h_n$ instead of $H_n$, we construct a uniformly distributed sequence
$\{z_n\}$ with
$$
\limsup_{N \to \infty} \frac{1}{N}\sum_{n=1}^N f(z_n)\le \int_\ast f.
$$
Clearly, $\lim_{n \to \infty} \rho(y_n, z_n) = 0$, therefore every sequence
$\{u_n\}$, where $u_n = y_n$ or $u_n = z_n$, is uniformly distributed.

Assume that $\{M_k\}$ is a sequence of positive integers with
$$
\lim_{k \to \infty} \frac{M_k}{\sum_{j=1}^k M_j}=1.
$$
Define the sequence $\{x_n\}$ as $x_n = y_n$ for
$\sum_{j=1}^{2k} M_j \le n < \sum_{j=1}^{2k+1} M_j$, and $x_n = z_n$ for $\sum_{j=1}^{2k+1} M_j \le n < \sum_{j=1}^{2k+2} M_j$. Then $\{x_n\}$ is uniformly distributed, and
$$
\liminf_{N \to \infty} \frac{1}{N}\sum_{n=1}^N f(x_n)\le \int_\ast f,
$$
$$
\limsup_{N \to \infty} \frac{1}{N}\sum_{n=1}^N f(x_n)\ge \int^\ast f.
$$
\end{proof}

Since $diam(K_j^n) \to 0$ uniformly for $n \to \infty$ we get
$$
cl(S) = \bigcap_{n=1}^\infty \bigcup_{S \cap L_j^n \neq \emptyset} cl(K_j^n).
$$
And so the $P$-continuity of $K_j^n$ implies
$\pi^\ast(S) = \int^\ast \mathcal{X}_S$. We get


\begin{corollary}
\label{cor:}
A set $S \subset \MM$ belongs to $\mathcal{Y}$ if and only if
for each sequence $\{x_n\}$ which is uniformly distributed the set $\A(\{x_n\},S)$ belongs to
$\D$ and in this case $d\big(\A(\{x_n\},S)\big)=\pi(S)$.
\end{corollary}

%
\section{The set of non-negative integers}\label{sec:set_of_non-negative_integers}

Assume that $\MM = \N$ is the set of non-negative integers.  Since $\N \subset \Omega$ we can consider
sequences of positive integers uniformly distributed in the compact space
$(\Omega, \rho, P)$.

A lot of examples of compact spaces containing the set $\N$ as a dense subset are described in the papers [IPT], [N], [N1], [P-P]. In this case the mappings $x:\N \to [0,1]$ are sequences.

The above corollary yields that a mapping $x:\N \to [0,1]$ is $\pi$-uniformly measurable if and only if for each sequence
of positive integers $\{k_n\}$ that is uniformly distributed in $\Omega$,
$$
\lim_{N \to \infty} \frac{1}{N} \sum_{n=1}^N \mathcal{X}_I(x(k_n)) = |I|,
$$
for any arbitrary subinterval $I \subset [0,1)$. Thus by the standard procedure we can prove:


\begin{theorem}
\label{theo:sequ_univ_meas_iff_integral}
A sequence $\{x(n)\}$ of elements of the unit interval is $\pi$-uniformly measurable if and only if for each sequence of positive integers $\{k_n\}$ thet is uniformly distributed in $\Omega$ we have
$$
\lim_{N \to \infty} \frac{1}{N} \sum_{n=1}^N f(x(k_n)) = \int_0^1 f(t)dt,
$$
for every Riemann integrable, (continuous,)  real function $f$ defined on $[0,1]$.
\end{theorem}

By the standard procedure we can derive the following:


\begin{proposition}
\label{prop:sequ_unif_measurable_iff_unif_distr_mod1}
 Let $\omega : \Omega \to [0,1]$ be a continuous function. Suppose that $\{n \}$ is uniformly distributed in $\Omega$. Then the sequence
$\{\omega(n)\}$ is $\pi$-uniformly measurable if and only it is uniformly distributed modulo $1$.
\end{proposition}

If we apply the proof of the main result in [TW] we obtain the part $1$ of following assertion. The part $2$ follows immediately from Theorem \ref{theo:existence _unif_measurab_mapping_NM}.

%
%

\begin{proposition}
\label{prop:omega}
  Let $\omega : \Omega \to [0,1]$ be a continuous function such that
$$
\int \omega \circ f = \int_0^1f(t)dt,
$$
for every real continuous function $f$ defined on $[0,1]$. Then
\begin{enumerate}
\item For every
sequence of positive integers $\{k_n \}$ that is uniformly distributed in $\Omega$ the sequence
$\{\omega(k_n)\}$ is uniformly distributed modulo $1$.

\item If $\{n\}$ is uniformly distributed in $\Omega$ then $\{\omega(n)\}$ is $\pi$-uniformly measurable.
\end{enumerate}
\end{proposition}

Analogously to uniform distribution, the notion of $\pi$-uniform measurability can be easily extended from the unit interval to an arbitrary compact metric space.
Let $\mathcal{M}$ be a compact metric space with Borel probability measure $\tilde{P}$. A mapping $y:\N \to \mathcal{M}$ is called $\pi$-uniformly measurable
if for every set of $\tilde{P}$-continuity $A \subset \mathcal{M}$ we have
$y^{-1}(A) \in \mathcal{Y}$ and $\pi(y^{-1}(A))=\tilde{P}(A)$.

Theorem \ref{theo:riemann_integrable_iff_exists_proper_limit} yields by the standard procedure, using Urysohn's \, Lemma, Weyl's  criterion for this case:

\begin{proposition}
\label{Weyl}
A  mapping $x : \N \to \mathcal{M}$ is $\pi$-uniformly measurable if and only if
$$
\lim_{N \to \infty}\frac{1}{N} \sum_{n=1}^N f(x(k_n))= \int f d\tilde{P}
$$
for every continuous real function defined on $\mathcal{M}$ and every sequence of positive integers $\{k_n\}$ that is uniformly distributed in $\Omega$.
\end{proposition}

Now we can show the existence property by applying the procedure from [H]. We start by proving the following proposition.

  \begin{proposition}
\label{y_unif_meas_then_x_unif_meas}
 Let $\pi$ be a density on the set of positive integers defined by \eqref{def_density}, having the Darboux property. If $y:\N \to [0,1]$ is a $\pi$-uniformly measurable mapping and $x:\N \to [0,1]$ is a mapping such that $\lim_{n \to \infty}|y(n)-x(n)|=0$, then also  $x$ is a $\pi$-uniformly measurable mapping.
\end{proposition}

\begin{proof}
The Darboux property of $\pi$ implies that $\pi(F)=0$ for each finite set $F\subset \N$. Thus for every sequence of positive integers
$\{k_n\}$ we have $\lim_{N \to \infty} \sum_{n \le N, k_n \in F} 1 =0$. If $f$ is a continuous real function defined on $[0,1]$ it is also uniformly continuous.
Thus for arbitrary $\varepsilon$ there exists $n_0$ such that for $n>n_0$ we have $|f(x(n))-f(y(n))| < \varepsilon$. This yields that for each sequence of positive integers $\{k_n\}$ that is uniformly distributed in $\Omega$ we have
$$
\limsup_{N \to \infty}\frac{1}{N}\sum_{n=1}^N  |f(x(k_n))-f(y(k_n))| \le \varepsilon,
$$
and the assertion follows.
\end{proof}


\begin{theorem}
\label{theo:existence _unif_measurab_mapping_NM}
Assume that $\pi$ is a density on $\N$ defined by \eqref{def_density}. If $\pi$ has the Darboux property on $\mathcal{Y}$ then for every compact metric space $\mathcal{M}$ with Borel probability measure
$\tilde{P}$ there exists a $\pi$-uniformly measurable mapping $x:\N \to \mathcal{M}$.
\end{theorem}

\begin{proof}
The compactness of $\mathcal{M}$ provides that there exists a system
$\mathcal{S}_m = \{S_m^j, j=1,\dots, r(m)\}$, $S_m^j \neq \emptyset$, $m=1,2,\dots$,
 of finite closed covers of $\mathcal{M}$, having the following properties:

i) If $n(\mathcal{S}_m)=\max \{diam(S_m^j), j=1,\dots, r(m)\}, m=1,2,\dots$, then  $$\lim_{m \to \infty}n(\mathcal{S}_m)=0.$$

ii) For every $m$ and $i \neq j$,  $\tilde{P}(S_m^i \cap S_m^j)=0$.

iii) If $m \le n$ then for every $i\le r(n)$ there exists a uniquely determined $j=j(i)\le r(m)$ such that
$S_n^i \subset S_m^j$.

iv) The indices in iii) are ordered in such a way that
$$i_1 < i_2 \Rightarrow j(i_1) \le j(i_2).$$

The property i) provides that

v) every real valued continuous function defined on $\mathcal{M}$ can be uniformly approximated by the linear combination of the functions $\mathcal{X}_{S_m^j}, j=1,\dots,r(m)$, for suitable $m$.

Moreover, for every $x \in \mathcal{M}$ and any open ball $B$ containing $x$ there exists a suitable $S_m^j$ such that
$x \in S_m^j \subset B$. Thus

vi) every open set is a union of a countable system of sets $S_m^j$.

To each cover $\mathcal{S}_m$ we can associate the finite system of intervals $I^j_m$,  for $j=1,\dots,r(m)$,
$I^j_m = \left[\sum_{l=1}^{j-1} \tilde{P}(S_m^l), \sum_{l=1}^{j} \tilde{P}(S_m^l)\right]$, and thus $|I^j_m| = \tilde{P}(S_m^l)$, and
 the endpoints of the intervals $I_m^{j}$, $j=1,\dots,r(m)$, form a division of the interval $[0,1]$.

Let $x:\N \to [0,1]$ be a $\pi$-uniformly measurable mapping, its existence being provided by Theorem 1.
The complement of the set of endpoints of the considered intervals is dense in the unit interval. Proposition 5 guarantees that for each dense subset of the unit interval a $\pi$-uniformly measurable mapping can be constructed such that values of this mapping belong to this set. Thus we can assume that the endpoints of $I^k_m,k=1,\dots,r(m)$, do not belong to the set $x(M)$.

 Each number $\alpha \in
[0,1]$ which does not coincide with the endpoints of the intervals $I^k_m,k=1,\dots,r(m)$, belongs to exactly one interval $I^{l(m)}_m$, for $m=1,2,\dots$, where $l(m):=l(m, \alpha)$ depends on $\alpha$.
Since for every $m=1,2,\dots$ we have $S^{l(m+1)}_{m+1} \subset S^{l(m)}_m$, the compactness of $\mathcal{M}$ guarantees that the set $D(\alpha) = \cap_{m=1}^\infty S^{l(m)}_m$ is not empty.
 Consider now a mapping $y:\N \to \mathcal{M}$ such that
$y(n) \in D(x(n)), n = 1,2,\dots$. We prove that this is a $\pi$-uniformly measurable mapping. Let
$\{k_n\}$ be an arbitrary sequence of positive integers that is uniformly distributed in $\Omega$. Denote, for
$H \subset \mathcal{M}$,
$$
P^\ast (H) = \limsup_{N \to \infty}\sum_{n=1}^{N}\mathcal{X}_H(x(k_n)),
$$
and
$$
P_\ast (H) = \liminf_{N \to \infty}\sum_{n=1}^{N}\mathcal{X}_H(x(k_n)).
$$
Let $\mathcal{Z}$ denote the system of all sets of the form $\cup_{t=1}^{T} S_m^{i_t}$. Clearly,
$$
\mathcal{X}_{\cup_{t=1}^{T} S_m^{i_t}}(y(n)) \ge \mathcal{X}_{\cup_{t=1}^{T} I_m^{i_t}}(x(n)).
$$
Thus, using the fact that the sequence $\{ x(k_n)\}$ is uniformly distributed in $[0,1]$, we obtain that
for every $S \in \mathcal{Z}$ we have $P_\ast (S) \ge \tilde{P}(S)$.

On the other hand, vi) implies that for every open set $G$ we have $\tilde{P}(G) = \sup\{\tilde{P}(S); S \subset G, S \in \mathcal{Z} \}$. Thus for every closed set $F$ and every $\varepsilon > 0$ there exists a set $S' \in \mathcal{Z}$ such that $S' \subset \mathcal{M} \setminus  F$ and $P(S') > \tilde{P}(\mathcal{M} \setminus  F) - \varepsilon$. Herefrom we obtain
$$
P^\ast (F) \le 1-P_\ast(S') \le 1-\tilde{P}(S') < 1 -\tilde{P}(\mathcal{M} \setminus  F) +\varepsilon =
\tilde{P}(F) +\varepsilon.
$$
Since $\varepsilon$ is arbitrary, we obtain $P^\ast (F) \le \tilde{P}(F)$. Therefore,
for every $S \in \mathcal{Z}$, $P_\ast(S)=P^\ast(S)=\tilde{P}(S)$ holds and the assertion follows from v).
\end{proof}

In the following we give examples that illustrate the situation in Theorem \ref{theo:existence _unif_measurab_mapping_NM}.
\begin{example}
With the notations in Theorem \ref{theo:existence _unif_measurab_mapping_NM}, let $\mathcal{M}=[0,1]$, and $\pi$ be a density on $\N$ that has the Darboux property. We consider $[0,1]$ to be endowed with the measure $\tilde{P}=\mu_r$, where $\mu_r$ is the binary measure of parameter $r$ on the unit interval, defined below. Thus in this case the equation \eqref{def_unif_measurable} that defines $\pi$-uniform measurability becomes $\pi\left(x^{-1}(I)\right)=\mu_r(I).$

Consider an arbitrarily fixed real number $r\in (0,1)$,  and  let $I_{0,0}=[0,1]$, $I_{n,j}=\left[\left.\frac{j}{2^n},\frac{j+1}{2^n} \right) \right.$, for $j=1,2,\dots 2^n-2$, and $I_{n,2^n-1}=\left[\frac{2^n-1}{2^n},1\right]. $ Then the binomial measure of parameter $r$ on $[0,1]$ is a probability measure that is defined by the conditions 
$\mu_r(I_{n+1,2j})=r\mu_r(I_{n,j})$ and $\mu_r(I_{n+1,2j+1})=(1-r)\mu_r(I_{n,j}),$
for $n=0,1,\dots$, and $j=0,1,\dots,2^n-1.$
For more details regarding the binomial distribution we refer, e.g., to \cite{okadasekiguchishiota_binomial, cristeaprodinger_jmaa}.

It is easy to see that $\mu_r(I_{j,2^n})=r^{n-k}(1-r)^k$, where $k$ is the number of digits $1$ in the binary digital expansion of the number $j$. Moreover, $\mu_r \left(\{y\}\right)=0$, for all $y\in [0,1]$.

Now we follow the steps and ideas from the proof of Theorem \ref{theo:unif_meas_exists_iff_darboux}. Under the assumption that $\pi$ has the Darboux property, we have that for each $B\in \mathcal{Y}$ there exists a set $B_1\in \mathcal{Y}$, $B_1 \subset B$ such that $\pi (B_1)=r\cdot \pi(B).$
We proceed as in the mentioned proof and obtain, for $n=1,2,\dots$, the decomposition consisting of the sets $U(j,2^n)$,  for $j=0,1,\dots, 2^n-1$, with $\pi\left( U(j,2^n)\right)=r^{n-k}(1-r)^k$, where where $k$ is the number of digits $1$ in the binary digital expansion of the number $j$. 
\end{example}

\noindent {\bf Remark.} The construction in Example 5 can also be extended analogously, following the above ideas, to the case when $\tilde P = \mu_{q,\mathbf{r}}$, where $\mu_{q,\mathbf{r}}$ is the multinomial measure of parameter vector $\mathbf{r}=(r_0,r_1,\dots,r_{q-1})$, $0<r_i<1$, $\sum_{i=1}^{q-1}r_i=1$, where the role the of the numeration base $2$ from Example 5 is here played by the positive integer  $q\ge 2$. 
For more details regarding this measure we refer, e.g., to \cite{okadasekiguchishiota_multinomial, cristeaprodinger_jmaa}.
Moreover, passing from the $n$-th decomposition of $\N$ to the $(n+1)$-th decomposition is given by the relation $\displaystyle U(j,q^n)=\cup_{k=0}^{q-1} U(qj+k,q^{n+1})$, for  $j=0,1,\dots, q^n-1$, where $n=1,2\dots$.

\begin{example}
With the notations in Theorem \ref{theo:existence _unif_measurab_mapping_NM}, let $\mathcal{M}$ be the well-known ``two thirds'' Cantor set $\mathcal{C}$, and $\pi$ be a density on $\N$ that has the Darboux property. We consider $\mathcal{C}$ to be endowed with the measure $\tilde{P}=\mu_{\mathcal{C}}$, where $\mu_{\mathcal{C}}$ is a probability measure on the Cantor set, defined inductively as follows.
Let us start with the set $C_0=[0,1]$ and let $\mu_{\mathcal{C}}(C_0)=1$. Let $C_1=[0,\frac{1}{3}]\cup[\frac{2}{3},1]$ and define $\mu_{\mathcal{C}}\left( [0,\frac{1}{3}]\right)=$$\mu_{\mathcal{C}}\left( [\frac{2}{3},0]\right)$$=\frac{1}{2}$. We proceed inductively and obtain at step $k\ge 1$ the set $C_k$ as the union of $2^k$ closed intervals of length $3^{-k}$ and each of them having the measure $\mu_{\mathcal{C}}$ equal to $2^{-k}$. Thus we obtain, in the limit, a measure whose support is the Cantor set. 

This measure coincides with the normalised Hausdorff measure on the Cantor set. For details regarding this measure, see, e.g., \cite{falconer_book1990, cristeatichy2003}
We consider the compact metric space $\mathcal{C}$ endowed with the Euclidean metric and the topology induced by the Euclidian topology of the real line. Then any set of the form $I \cap \mathcal{C}$, where $I \subset [0,1]$ is an interval, is a set of $\mu_{\mathcal{C}}$-continuity.

 In this case the definition of $\pi$-uniform measurability that we mentioned on page 10 becomes $\pi\left(x^{-1}(A)\right)=\mu_{\mathcal{C}}(A)$, for every set A of $\mu_{\mathcal{C}}$-continuity. 
\end{example}

\end{document}